\documentclass[oneside, reqno] {amsart}
\usepackage{graphicx} % Required for inserting images
\usepackage{amsmath}
\usepackage{amssymb}
\usepackage{euscript}
\usepackage{empheq}
\usepackage{amsthm}
\usepackage{mathtools}
\usepackage[dvipsnames]{xcolor}
\usepackage{xcolor}
\usepackage{enumitem}
\usepackage{cite} %Imports biblatex package
\newtheorem{theorem}{Theorem}[section]

\newtheorem{lemma}[theorem]{Lemma}
\numberwithin{equation}{section}
\newtheorem{proposition}{Proposition}[section]

\theoremstyle{definition}
\newtheorem{definition}{Definition}[section]
\newtheorem{remark}{Remark}[section]

\title{Singularity formation for the 1D model of EMHD}

\begin{document}
\author{Chao Wu}
\address{Department of Mathematics, Statistics and Computer Science, University of Illinois at Chicago, Chicago, IL 60607, USA}
\email{cwu206@uic.edu}
\begin{abstract}
    In this paper, we study the singularity formation phenomenon of the 1D model of Electron Magnetohydrodynamics (EMHD) given by (\ref{the model}). We will construct a solution whose $C^3$-norm blows up in finite time. Then, we will show that the solution is in $C^{\infty}(\mathbb{R}\backslash \{0\})\cap C^{3,s}(\mathbb{R})\cap H^3(\mathbb{R})$ and is not asymptotically self-similar.
\end{abstract}
\allowdisplaybreaks
\counterwithin{equation}{section}
\nopagebreak
\maketitle

\section{Introduction}
\subsection{An overview}
We consider the 1 model of EMHD given by
\begin{equation}\label{the model}
    \begin{cases}
        B_t=-2bJB_x+b HB_{xx}B\\
        B_x=HJ,
    \end{cases}
\end{equation}
where $b>0$ and $H$ is the Hilbert transform on $\mathbb{R}$. The study of 1D toy models for understanding the original equations dates back to work on the Euler equations. For example, Constantin, Lax and Majda proposed a 1D model for the Euler equation \cite{constantin1985simple}. The model is given by
\begin{align*}
    \omega_t&=\omega H\omega\\
    u_x&=H\omega.
\end{align*}The authors solved the model exactly and showed the singularity formation phenomenon for a class of initial data. Then De Gregorio proposed the De Gregorio model \cite{de1990one, de1996partial}. The model is given by
\begin{align*}
    \omega_t+u\omega_x&=\omega H\omega\\
    u_x&=H\omega
\end{align*}
which is different from the Constantin-Lax-Majda model by the appearance of the transport term. The numerical results of \cite{de1990one, de1996partial} showed some evidence that a finite-time blow-up may not occur. This suggests that the transport term has some smoothing effects. In order to understand the competition between the transport term and the stretching term, Okamoto, Sakajo and Wunsch \cite{okamoto2008generalization} suggested to study the following model
\begin{align*}
    \omega_t+au\omega_x&=\omega H\omega\\
    u_x&=H\omega,
\end{align*}
and conjectured the global in time existence of a solution for the De Gregorio model. Later, Jia, Stewart and \v Sver\`k \cite{jia2019gregorio} proved that for the De Gregorio model, solutions with initial data near a steady state are global and converge to that steady state. In contrast, Elgindi and Jeong \cite{elgindi2020effects} showed the phenomenon of singularity formation for initial data in H\"older spaces. They also showed a self-similar blow-up for the Okamoto-Sakajo-Wunsch model with small $|a|$. Later, Elgindi, Ghoul, and Masmoudi \cite{elgindi2021stable} further showed that such a self-similar blow-up is stable.

Following the footprint of the study on the 1D models of fluid equations, several 1D models of magnetohydrodynamics were proposed. For example, Dai, Vyas and Zhang \cite{Dai2023} proposed a 1D model for the 3D incompressible ideal magnetohydrodynamics. The model is given by
\begin{align*}
    \Omega_t+m\Omega_x-\omega p_x&=0,\\
    \omega_t+ap\omega_x-\Omega m_x&=0,\\
    p_x&=H\Omega,\\
    m_x&=H\omega.
\end{align*}
The authors of \cite{Dai2023} showed the local well-posedness and established the Beale-Kato-Majda regularity criterion. Also, the global existence was obtained without the stretching effect. Later, Dai \cite{dai2022reducedmodelselectronmagnetohydrodynamics} proposed a family of simplified non-local non-linear models for EMHD, which are known as the 1D models of EMHD. The models are given by
\begin{align*}
    B_t+aJB_x&=b HB_{xx}B+\mu\Lambda^{\alpha}B,\quad a,b\in\mathbb{R},\,\alpha>0\\
        B_x&=HJ.
\end{align*}
Comparing with the original EMHD given by
\begin{equation}
\begin{split}
B_t+ \nabla\times ((\nabla\times B)\times B)=&\ \nu\Delta B\\
\nabla\cdot B=&\ 0,
\end{split}
\end{equation}
the spatial geometric features are suppressed in the 1D simplifications, but it still captures the two important nonlinear effects: transport and stretching. In addition, it is believed that the study of the 1D model is helpful in studying the axisymmetric solution to the electron MHD. The reader is referred to \cite{dai2025wellposednessblowup1delectron} for more details. Compared with the generalized De Gregorio model, the nonlinearities of the 1D model of EMHD are more singular, which makes the study of some topics such as local wellposedness more difficult. Despite this, local wellposedness results can still be expected with dissipation $\Delta^{\alpha}B$. In \cite{dai2022reducedmodelselectronmagnetohydrodynamics}, the author gave a local wellposedness result in the regime $\alpha\in(2,\infty)$ on torus with initial data in Sobolev spaces. Recently, the exponent of dissipation was improved to $\alpha>1$ \cite{dai2025wellposednessblowup1delectron}. The singularity formation of the 1D model has also been studied recently \cite{dai2022reducedmodelselectronmagnetohydrodynamics, dai2025wellposednessblowup1delectron}. In this paper, we continue to study the singularity formation phenomena.

\subsection{The main result}
The main result of the paper is the following:
\begin{theorem}
    For any $b\neq0$ and $\epsilon>0$, there exist $s\in(0,1)$, $T>0$, $C>0$ and a solution to the 1D model of (\ref{the model}) in the time interval $[-T,0)$ such that for all $t$ in this interval, the solution $B(\cdot,t)\in C^{\infty}(\mathbb{R}\backslash \{0\})\cap C^{3,s}(\mathbb{R})\cap H^3(\mathbb{R})$ and such that $\frac{1}{C|t|}\leq\max|\partial_x^3B(\cdot,t)|\leq\frac{C}{|t|}$. Moreover, $B_{n,k}(x,t)$ is not asymptotically self-similar.
\end{theorem}

\begin{remark}
    At present, we do not have a definition of an asymptotically self-similar solution. We will show that there is no asymptotic self-similar profile for (\ref{the model}).
\end{remark}

\subsection{Technical lemmas}
In this subsection, we collect several technical lemmas that are needed. The proof can be found in \cite{cordoba2023finitetimesingularities3d}. The blow-up mechanism relies on the following ODE system. Interested readers are referred to \cite{cordoba2023finitetimesingularities3d} for more details.
\begin{lemma}\label{the first preparing lemma}
    Let $\lambda>0$, $\delta<1$, and assume that
    \begin{gather*}
        x_n(t)=A^{n},\quad \dot{x}_n(t)=x_n(t)\sum_{j=0}^{n-1}a_j(t)x_j(t)
    \end{gather*}
    where $A<1$ and $\delta\leq a_j$. Let $a>0$ solve $a=A(1-e^{-a}).$ Then we have
    \begin{gather*}
        \ln{A}<a<2(A-1)\quad\text{and}\quad\int_{-a}^0x_n(t)dt\leq \frac{a}{\delta}
    \end{gather*}
\end{lemma}
\begin{remark}\label{tech remark}
    If $\delta\leq a_j$ becomes $\delta\geq a_j$, then the same argument can be applied to show that the inequality is reversed.
\end{remark}
\begin{lemma}\label{increasing lemma}
    For $n>k$, the ratio $x_n(t)/x_k(t)$ increases in t.
\end{lemma}

\subsection{Outline of the paper}
The rest of the paper is organized as follows. Section 2 starts with giving the form of blow-up followed by showing local existence, and then gives the solution as the limit of a series. Section 3 shows several properties of the solution.

\section{Construction of the solution}
\subsection{The profile}
In this section, we consider the following profile for $B$.
Let
\begin{align*}
    B_{n,k}(x,t)=x_{n,k}^c(t)\left(\frac{r}{A}\right)^{dk}W_{n,k}\left(\frac{A^kx}{x_{n,k}(t)r^k},t\right)
\end{align*}
where $c,d$ are some numbers and $x_{n,k}:\mathbb{R}\rightarrow\mathbb{R}$ satisfies some ODE. Then note that
\begin{align*}
    &\partial_tB_{n,k}\\
    =&cx_{n,k}^{c-1}\dot{x}_{n,k}(t)\left(\frac{r}{A}\right)^{dk}W_{n,k}\left(\frac{A^kx}{x_{n,k}r^k},t\right)-x_{n,k}^{c-2}(t)\left(\frac{r}{A}\right)^{(d-1)k}x\partial_xW_{n,k}\left(\frac{A^kx}{x_{n,k}r^k},t\right)+\\
    &+x_{n,k}^c(t)\left(\frac{r}{A}\right)^{dk}\partial_tW_{n,k}\left(\frac{A^kx}{x_{n,k}r^k},t\right)
\end{align*}
Denote
\begin{align*}
    B_n=\sum_{k=0}^{n}B_{n,k}
\end{align*}
and require $B_n$ to be a solution to (\ref{the model}). Then we get the PDE for the profile $W_{n,k}$, which reads
\begin{align}
    \partial_tW_{n,k}=&\left(2b\frac{1}{x_{n,k}}\left(\frac{A}{r}\right)^k\partial_xHB_{n}\left(xx_{n,k}(t)\left(\frac{r}{A}\right)^k,t\right)+\frac{x\dot{x}_{n,k}}{x_{n,k}}\right)\partial_xW_{n,k}+\notag\\
    &+\left(b\partial_x^2HB_{n}\left(xx_{n,k}(t)\left(\frac{r}{A}\right)^k,t\right)-c\frac{\dot{x}_{n,k}}{x_{n,k}}\right)W_{n,k}
\end{align}
By the standard energy estimate, we have that for $N\in\mathbb{N}$
\begin{align*}
    \frac{1}{2}\frac{d}{dt}\|W_{n,k}\|_{\dot{H}^N}^2=-\left(\frac{1}{2}+c\right)\frac{\dot{x}_{n,k}}{x_{n,k}}\|W_{n,k}\|_{H^N}^2,
\end{align*}
so the local wellposedness is trivially obtained both forward in time and backward in time for proper $x_k(t)$.

\subsection{The form of blow-up}
Inspired by the novel blow-up mechanism in \cite{cordoba2023finitetimesingularities3d}, we plan to construct a solution $B$ of the form
\begin{align}
    B(x,t)=\lim_{n\to\infty}\sum_{k=0}^{n}B_{n,k}(x,t),
\end{align}
where we require
\begin{align}\label{eq for omega nk}
    \partial_tB_{n,k}=-2bJ_{n}\partial_xB_{n,k}+bH\partial_x^2B_nB_{n,k}.
\end{align}
Now we consider the data
$$B_n(x,0):=\sum_{k=k}^{n}B_{n,k}(x,0),\quad B_{n,k}(x,0)=r^{3k}A^k\phi\left(\frac{x}{r^k}\right)$$
and define
\begin{align}\label{solution series}
    B_n(x,t)=\sum_{k=0}^{n}B_{n,k}(x,t),
\end{align}
so $B_n$ is a solution to
\begin{align}\label{aucxilly equation}
    \partial_tB_{n}=-2J_{n}\partial_xB_{n}+B_{n}H\partial_x^2B_n.
\end{align}
We denote
$$B_-(x,t):=\sum_{j<k}B_{n,j}(x,t)\quad\text{and}\quad B_+(x,t):=\sum_{j>k}B_{n,j}(x,t)$$\newline
Now we introduce the relation between $B_{n,k}(x,t)$ and its profile. Let
\begin{align}\label{def omega n}
    &B_{n,k}(x,t):=x_{n,k}^{4}(t)\left(\frac{r}{A}\right)^{3k}W_{n,k}\left(\frac{A^{k}x}{x_{n,k}(t)r^k},t\right)
\end{align}
so
$$W_{n,k}(x,t)=\frac{A^{3k}B_{n,k}\left(\frac{x_{n,k}(t)r^kx}{A^{k}},t\right)}{r^{3k}x_{n,k}^{4}(t)}$$
where we take the initial data
$$x_{n,k}(0)=A^{k},$$
so
$$\quad W_{n,k}(x,0)=\frac{1}{r^{3k}A^k}B_{n,k}\left(xr^k,0\right)=\phi(x).$$
Note that (\ref{eq for omega nk}) and (\ref{def omega n}) together give the evolution for $W_{n,k}(t)$ which reads
\begin{align}
    \partial_tW_{n,k}=&\left(2b\frac{1}{x_{n,k}}\left(\frac{A}{r}\right)^k\partial_xHB_{n}\left(xx_{n,k}(t)\left(\frac{r}{A}\right)^k,t\right)+\frac{x\dot{x}_{n,k}}{x_{n,k}}\right)\partial_xW_{n,k}+\notag\\
    &+\left(b\partial_x^2HB_{n}\left(xx_{n,k}(t)\left(\frac{r}{A}\right)^k,t\right)-4\frac{\dot{x}_{n,k}}{x_{n,k}}\right)W_{n,k}
\end{align}
We define
\begin{align}\label{at 0 is 0}
    \dot{x}_{n,k}(t)&=x_{n,k}(t)H\partial_x^2B_-(0,t)\\
    &:=x_{n,k}(t)\sum_{j=0}^{k-1}H\partial_x^2W_{n,j}(0,t)x_{n,j}(t)\notag.
\end{align}
Passing to the limit $n\to\infty$ for (\ref{solution series}), we will have the unbounded (in $C^3(\mathbb{R})$) data
\begin{align*}
    B(x,0)=\sum_{k=0}^{\infty}\left(r^3A\right)^kW_k(x/r^k,0),
\end{align*}
where $W_k$ is defined to be the limit of the sequence $\{W_{n,k}\}_{n\geq1}$ in some sense. Since there is no a prior local well-posedness result for the unbounded data, we plan to start from bounded data
\begin{align*}
    B_{n,k}(x,0)=\sum_{k=0}^{n}(r^3A)^{k}W_{n,k}(x/r^k,0)
\end{align*}
and use the evolution equation of $W_{n,k}(x,t)$ to show a lifespan of $W_{n,k}(x,t)$ independent of $n.$
\subsection{bootstrapping ansatz}
We now state the bootstrapping assumptions: for $t\in[-T,0]$ and $k=0,\dots,n$,
\begin{enumerate}
    \item $W_{n,k}(\cdot,t)=0$ outside
    $[-1-2r,-1+2r]\cup[1-2r,1+2r].$
    \item $\|W_{n,k}(\cdot,t)-\phi\|_{\dot{H}^{4}}\leq\epsilon$, where $\epsilon>0$ is small.
\end{enumerate}
\subsection{Estimates}
\begin{lemma}\label{bootstrap assumptions implication}
    The bootstrapping assumptions yield
    \begin{align*}
        \|W_{n,k}-\phi\|_{L^1}\leq512r^4\sqrt{2r}\epsilon
    \end{align*}
\end{lemma}
\begin{proof}
    Let $f_k=W_{n,k}-\phi_k$. Then by the Poincar\'e inequality,
    \begin{align*}
        \|f\|_{L^1}\leq(4r)^4\|\partial_x^4f\|_{L^1}\leq2^9r^4\sqrt{2r}\|\partial_x^4f\|_{L^{2}}\leq C\epsilon.
    \end{align*}
\end{proof}
\begin{lemma}\label{control aj}
    Lemma \ref{bootstrap assumptions implication} immediately implies $$|\partial_x^2HW_{n,k}(0,t)|\sim H\phi''(0).$$
\end{lemma}
\begin{proof}
    We have
    \begin{align*}
        &|H\partial_x^2W_{n,k}(0)-H\phi_k''(0)|\\
        =&\left|\int_{\text{supp}f_k}\frac{W_{n,k}(y)-\phi(y)}{y^3}dy\right|\\
        \leq&\frac{C}{(1-2r)^3}\epsilon
    \end{align*}
\end{proof}
Now we estimate $\partial_x^2HB_{n,k}$, $\partial_x^2HB_{-}$, and $\partial_x^2HB_{+}$.
\begin{lemma}\label{for closing bootstrap}
    There are constants $C_i(r,\epsilon)$ ($i=1,\cdots,8$) such that for $A\in(\frac{1}{2},1)$, and $x\in\text{supp}\,B_{n,k}(\cdot,t)$, namely, $(|x|/x_{n,k}(t))$$(A/r)^k\in[1-2r,1+2r]$,
    \begin{align*}
        \left|\partial_xHB_-(x,t)\right|
        &\leq C_1(r,\epsilon)x_{n,k}(t)\left(\frac{r}{A}\right)^k\\
        \left|\partial_x^2HB_-(x,t)\right|
        &\leq C_2(r,\epsilon)\\
        \left|\partial_x^3HB_-(x,t)\right|&\leq C_3 A^k\\
        \left|\partial_x^4HB_-(x,t)\right|&\leq C_4 \left(\frac{A}{r}\right)^k\\
        \left|\partial_x^5HB_{-}(x,t)\right|&\leq C_5(r,\epsilon)\left(\frac{A}{r}\right)^{2k}\frac{1}{x_{n,k}(t)}\\
        \left|\partial_x^5HB_+(x,t)\right|&\leq C_6(r,\epsilon)\left(\frac{A}{r}\right)^{2k}\frac{1}{x_{n,k}(t)}\\
        \left|\partial_x^4HB_+(x,t)\right|&\leq C_7(r,\epsilon)\left(\frac{A}{r}\right)^k\\
        \left|\partial_x^3HB_{+}(x,t)\right|&\leq C_8(r,\epsilon)A^k\\
        |\partial_x^2HB_{+}(x,t)|&\leq C_9(r,\epsilon)\\
        |\partial_xHB_{+}(x,t)|&\leq C_{10}(r,\epsilon)(Ar)^{k}
    \end{align*}
\end{lemma}
\begin{proof}
    By Lemma \ref{bootstrap assumptions implication}, we have
    \begin{align*}
        \|W_{n,k}\|_{L^1}\leq\|\phi\|_{L^1}+C\epsilon.
    \end{align*}
    Then for $x\in[0,1-2r]$, we have
    \begin{align*}
        &\partial_x^N(HW_{n,k}(x,t))\\
        \leq&\frac{1}{\pi}\left(\int_{\text{supp}(W_{n,k})}\frac{|W_{n,k}(y)|}{\pi|(x-y)|^{N+1}}dy\right)\\
        \leq&\int_{\text{supp}(W_{n,k})}\left|\frac{W_{n,k}(y)}{\pi|1-2r-x|^{N+1}}\right|dy\\
        \leq&\frac{\|\phi\|_{L^1}+C\epsilon}{\pi|1-2r-x|^{N+1}}.
    \end{align*}
    For $j<k$ and
    $$|x|\in\left[(1-2r)|x_{n,k}|\left(\frac{r}{A}\right)^{k},(1+2r)|x_{n,k}|\left(\frac{r}{A}\right)^{k}\right],$$
    we observe that
    \begin{align*}
        \frac{A^{j}|x|}{x_{n,j}r^j}
        &\in\left[(1-2r)\frac{x_{n,k}}{x_{n,j}}\left(\frac{r}{A}\right)^{k-j},(1+2r)\frac{x_{n,k}}{x_{n,j}}\left(\frac{r}{A}\right)^{k-j}\right]\\
        &\subset\left[0,(1+2r)r^{k-j}\right]\\
        &\subset[0,(1+2r)r]\\
        &\subset[0,1-2r].
    \end{align*}
    So we have
    \begin{align}
        H\partial_x^NB_{n,j}(x,t)
        &=x_{n,j}^{4-N}(t)\left(\frac{r}{A}\right)^{(3-N)j}\partial_x^NHW_{n,j}\left(\frac{A^{j}x}{x_{n,j}r^j},t\right)\notag\\
        &\leq\left(r^{3-N}A\right)^{j}\frac{\|\phi\|_{L^1}+C\epsilon}{\pi\left|1-2r-(1+2r)r\right|^4}\label{to be summed}
    \end{align}
    Summing (\ref{to be summed}) over $j=0,1,\dots,k-1$, we get
    \begin{align}\label{the above inequality}
        H\partial_x^2B_-(x,t)
        &\leq C_2(r,\epsilon)\\
        H\partial_x^3B_-(x,t)&\leq C_3 A^k\\
        H\partial_x^4B_-(x,t)&\leq C_4 \left(\frac{A}{r}\right)^k
    \end{align}
    Integrating from 0 to $|x|\leq(1+2r)x_{n,k}(t)\left(\frac{r}{A}\right)^k$ and using the fact that $B_{n,j}$ are odd yield the first inequality.
    For $N=5$, we note that
    \begin{align*}
        \partial_x^5HB_{n,j}(x,t)
        &=\frac{1}{x_{n,j}(t)}\left(\frac{A}{r}\right)^{2j}\partial_x^5HW_{n,j}\left(\frac{A^jx}{x_{n,j}r^j},t\right)\\
        &=\frac{x_{n,k}(t)}{x_{n,j}(t)}\left(\frac{A}{r}\right)^{2j}\partial_x^5HW_{n,j}\left(\frac{A^jx}{x_{n,j}r^j},t\right)\frac{1}{x_{n,k}(t)}\\
        &\leq A^k\left(\frac{A}{r^2}\right)^j\partial_x^5HW_{n,j}\left(\frac{A^jx}{x_{n,j}r^j},t\right)\frac{1}{x_{n,k}(t)},
    \end{align*}
    summing over $j=0,\cdots,k-1$ yields
    \begin{align*}
        \partial_x^5HW_{-}(x,t)&\leq C_5(r,\epsilon)\left(\frac{A}{r}\right)^{2k}\frac{1}{x_{n,k}(t)}
    \end{align*}
    For $|x|>1+2r$, we have that
    \begin{align}
        &HW_{n,j}(x,t)\notag\\
        =&\frac{1}{\pi}\int_0^x\partial_yHW_{n,j}(y,t)dy\notag\\
        =&\frac{1}{2\pi}\int_{-x}^x\int_{\text{supp}W_{n,j}}\frac{HW_{n,j}(z,t)}{(y-z)^2}dzdy\notag\\
        =&\frac{1}{2\pi}\int_{\text{supp}W_{n,j}}HW_{n,j}(z,t)\left(\frac{1}{x-z}-\frac{1}{-x-z}\right)dz\label{inductively}\\
        \leq&\frac{\|\phi\|_{L^1}+C\epsilon}{2\pi}\left(\frac{1}{|x|-1-2r}-\frac{1}{|x|+1+2r}\right)\notag
    \end{align}
    By (\ref{inductively}), we have
    \begin{align*}
        \partial_xHW_{n,j}(x,t)&\leq\frac{2\left(\|\phi\|_{L^1}+C\epsilon\right)(1+2r)|x|}{\pi(|x|^2-(1+2r)^2)^2}\\
        \partial_x^2HW_{n,j}(x,t)&\leq\frac{3\left(\|\phi\|_{L^1}+C\epsilon\right)(1+2r)|x|^2}{\pi(|x|^2-(1+2r)^2)^3}\\
        \partial_x^3HW_{n,j}(x,t)&\leq\frac{4\left(\|\phi\|_{L^1}+C\epsilon\right)(1+2r)|x|^3}{\pi(|x|^2-(1+2r)^2)^4}\\
        \partial_x^4HW_{n,j}(x,t)&\leq\frac{5\left(\|\phi\|_{L^1}+C\epsilon\right)(1+2r)|x|^4}{\pi(|x|^2-(1+2r)^2)^5}\\
        \partial_x^5HW_{n,j}(x,t)&\leq\frac{6\left(\|\phi\|_{L^1}+C\epsilon\right)(1+2r)|x|^5}{\pi(|x|^2-(1+2r)^2)^6}
    \end{align*}
    For $j>k$ and
    \begin{align*}
        |x|\in\left[(1-2r)|x_{n,k}|\left(\frac{r}{A}\right)^{k},(1+2r)|x_{n,k}|\left(\frac{r}{A}\right)^{k}\right]
    \end{align*}
    we observe that
    \begin{align*}
        \frac{A^{j}|x|}{x_{n,j}r^j}
        &\in\left[(1-2r)\frac{x_{n,k}}{x_{n,j}}\left(\frac{r}{A}\right)^{k-j},(1+2r)\frac{x_{n,k}}{x_{n,j}}\left(\frac{r}{A}\right)^{k-j}\right]\\
        &\subset\left[(1-2r)r^{k-j},\infty\right)\\
        &\subset[(1-2r)r^{-1},\infty]\\
        &\subset[1+2r,\infty],
    \end{align*}
    and that
    \begin{align*}
        \frac{A^j|x|}{x_{n,j}r^j}\geq\frac{(1-2r)A^{j-k}x_{n,k}}{x_{n,j}r^{j-k}}\geq(1-2r)r^{k-j},
    \end{align*}
    which implies
    \begin{align*}
        \left(\frac{A^j|x|}{x_{n,j}r^j}\right)^2-(1+2r)^2\geq\frac{7}{16}\left(\frac{A^j|x|}{x_{n,j}r^j}\right)^2\geq\frac{7}{16}(1-2r)^2r^{2(k-j)}
    \end{align*}
    so we have, for example when $N=4$
    \begin{align*}
        &|\partial_x^4H\omega_{n,j}(x,t)|\\
        =&\left(\frac{A}{r}\right)^j\left|\partial_x^4HW_{n,j}\left(\frac{A^jx}{x_{n,j}r^j},t\right)\right|\\
        \leq&\left(\frac{A}{r}\right)^j\frac{4\left(\|\phi\|_{L^1}+C\epsilon\right)(1+2r)\left(\frac{16}{7}\right)^{\frac{3}{2}}\left|\left(\frac{A^j|x|}{x_{n,j}r^j}\right)^2-(1+2r)^2\right|^{\frac{3}{2}}}{\pi\left(\left(\frac{A^j|x|}{x_{n,j}r^j}\right)^2-(1+2r)^2\right)^4}\\
        =&\left(\frac{A}{r}\right)^j\frac{4\left(\|\phi\|_{L^1}+C\epsilon\right)(1+2r)\left(\frac{16}{7}\right)^{\frac{3}{2}}}{\pi\left(\left(\frac{A^j|x|}{x_{n,j}r^j}\right)^2-(1+2r)^2\right)^{\frac{5}{2}}}\\
        \leq&\left(\frac{A}{r}\right)^j\frac{4\left(\|\phi\|_{L^1}+C\epsilon\right)(1+2r)\left(\frac{16}{7}\right)^{\frac{3}{2}}}{\pi\left(\frac{7}{16}(1-2r)^2r^{2(k-j)}\right)^{\frac{5}{2}}}\\
        =&\left(\frac{A}{r}\right)^j\frac{4\left(\|\phi\|_{L^1}+C\epsilon\right)(1+2r)\left(\frac{16}{7}\right)^{4}}{\pi\left((1-2r)^2\right)^{\frac{5}{2}}}r^{5(j-k)},\\
    \end{align*}
    summing over $j\geq k+1$ yields
    \begin{align*}
        \left|\partial_x^4H\omega_{+}(x,t)\right|\leq\frac{4\left(\|\phi\|_{L^1}+C\epsilon\right)(1+2r)\left(\frac{16}{7}\right)^{4}Ar^4}{\pi\left((1-2r)^2\right)^{\frac{5}{2}}(1-Ar^4)}\left(\frac{A}{r}\right)^k:=C_9(r,\epsilon)\left(\frac{A}{r}\right)^k.
    \end{align*}
    The cases corresponding to other values of $N$ are estimated similarly so we omit the details.
\end{proof}
\subsection{Evolution of $W_{n,k}$}
In this subsection, we show that $W_{n,k}$ is close to $\phi$ in 
$\dot{H}^4$. Recall that the evolution of $W_{n,k}$ is given by
\begin{align*}
    \partial_tW_{n,k}=&\left(2b\frac{1}{x_{n,k}}\left(\frac{A}{r}\right)^k\partial_xHB_{n}\left(xx_{n,k}(t)\left(\frac{r}{A}\right)^k,t\right)+\frac{x\dot{x}_{n,k}}{x_{n,k}}\right)\partial_xW_{n,k}+\notag\\
    &+\left(b\partial_x^2HB_{n}\left(xx_{n,k}(t)\left(\frac{r}{A}\right)^k,t\right)-3\frac{\dot{x}_{n,k}}{x_{n,k}}\right)W_{n,k}
\end{align*}
\begin{lemma}\label{for higher energy estimates}
    There are constants $C_i(r,\epsilon)$ ($i=11,12$) such that for $A\in\left(\frac{1}{2},1\right)$ and $x\in\text{supp }B_{n,k}$,
    \begin{align*}
        |\partial_x^NH\partial_x^3B_-(x,t)|&\leq C_{11}(r,\epsilon)\left(\frac{A}{r}\right)^{Nk}\left(\frac{1}{x_{n,k}(t)}\right)^{N-1}\\
        |\partial_x^NH\partial_x^3B_+(x,t)|&\leq C_{12}(r,\epsilon)\left(\frac{A}{r}\right)^{Nk}\left(\frac{1}{x_{n,k}(t)}\right)^{N-1}
    \end{align*}
\end{lemma}
\begin{proof}
    The proof is the same as that of Lemma \ref{for closing bootstrap} with taking more derivatives.
\end{proof}
\begin{lemma}
    There are constants $C(N)$ such that for $A\in\left(\frac{1}{2},1\right)$ and $x\in\text{supp}W_{n,k}$,
    \begin{align*}
        \|W_{n,k}\|_{\dot{H}^N}\leq C(N)
    \end{align*}
\end{lemma}
\begin{proof}
    The standard energy estimate shows that
    \begin{align*}
        \frac{d}{dt}\|W_{n,k}\|_{\dot{H}^N}^2\leq C(N)\left|\frac{\dot{x}_{n,k}}{x_{n,k}}\right|\|W_{n,k}\|_{\dot{H}^N}^2=C(N)\partial_x^2HB_-(0,t)\|W_{n,k}\|_{\dot{H}^N}^2
    \end{align*}
    Then the proof is a direct consequence of Lemma \ref{for closing bootstrap}.
\end{proof}
Now we are ready to show that $W_{n,k}$ is close to $\phi$ in $\dot{H}^4$. Denote
\begin{align*}
    E_{n,k}(t)=\int_{\mathbb{R}}\left|\partial_x^4(W_{n,k}-\phi)\right|^2dx.
\end{align*}
Then
\begin{align*}
    \frac{d}{dt}E_{n,k}(t)=2\int_{\mathbb{R}}\partial_x^8(W_{n,k}-\phi)\partial_tW_{n,k}dx=2(E_1+\cdots+E_8)
\end{align*}
where
\begin{align*}
    E_1&=2b\frac{1}{x_{n,k}}\left(\frac{A}{r}\right)^k\int\partial_x^8(W_{n,k}-\phi)\partial_xHB_n\left(xx_{n,k}(t)\left(\frac{r}{A}\right)^k,t\right)\partial_x(W_{n,k}-\phi)dx\\
    E_2&=2b\frac{1}{x_{n,k}}\left(\frac{A}{r}\right)^k\int\partial_x^8(W_{n,k}-\phi)\partial_xHB_n\left(xx_{n,k}(t)\left(\frac{r}{A}\right)^k,t\right)\phi'dx\\
    E_3&=\int\partial_x^8(W_{n,k}-\phi)\frac{x\dot{x}_{n,k}(t)}{x_{n,k}(t)}\partial_x(W_{n,k}-\phi)dx\\
    E_4&=\int\partial_x^8(W_{n,k}-\phi)\frac{x\dot{x}_{n,k}(t)}{x_{n,k}(t)}\phi'dx\\
    E_5&=\int\partial_x^8(W_{n,k}-\phi)b\partial_x^2HB_{n}\left(xx_{n,k}(t)\left(\frac{r}{A}\right)^k,t\right)(W_{n,k}-\phi)dx\\
    E_6&=\int\partial_x^8(W_{n,k}-\phi)b\partial_x^2HB_{n}\left(xx_{n,k}(t)\left(\frac{r}{A}\right)^k,t\right)\phi dx\\
    E_7&=-3\frac{\dot{x}_{n,k}}{x_{n,k}}\int\partial_x^8(W_{n,k}-\phi)(W_{n,k}-\phi)dx\\
    E_8&=-3\frac{\dot{x}_{n,k}}{x_{n,k}}\int\partial_x^8(W_{n,k}-\phi)\phi dx
\end{align*}
By auxiliary lemmas in Appendix A, we have
\begin{align*}
    \left|\frac{d}{dt}\sqrt{E_{n,k}}\right|\leq C_{13}(r, \epsilon)\sqrt{E_{n,k}}+C_{14}(r,\epsilon).
\end{align*}
Note that the coefficients are independent of $n$ and $k$. Solving it for $t<0$ with $E_{n,k}=0$ gives
\begin{align*}
    \sqrt{E_{n,k}}\leq C_{14}\int_t^0e^{\int_t^sC_{13}dr}ds\leq C_{14}|t|e^{C_{13}|t|}.
\end{align*}
Set $t=-T_1$, where $T_1$ is chosen sufficiently close to 0, then we close the first bootstrapping ansatz. To show that the support is within $[-1-2r,-1+2r]\cup[1-2r,1+2r]$, we only need to show that the maximum distance traveled by the end points is at most $r$. Note that the speed of the end points is
\begin{align*}
    2b\frac{1}{x_{n,k}}\left(\frac{A}{r}\right)^k\partial_xHB_{n}\left(xx_{n,k}(t)\left(\frac{r}{A}\right)^k,t\right)+\frac{x\dot{x}_{n,k}}{x_{n,k}}.
\end{align*}
Then the end points move at most
\begin{align*}
    2b\int_{-T_2}^0\frac{1}{x_{n,k}}\left(\frac{A}{r}\right)^k\partial_xHB_{n}\left(xx_{n,k}(t)\left(\frac{r}{A}\right)^k,t\right)+\frac{x\dot{x}_{n,k}}{x_{n,k}}dt\leq C|T_2|.
\end{align*}
where we used lemma \ref{for closing bootstrap}, the compactly supported ansatz for $W_{n,k}$, and the fact that $x_{n,k}(t)\leq A^k$. Choosing $T_2$ sufficiently close to 0
completes the second bootstrapping ansatz. Setting $T:=\min(T_1,T_2)$, we get a lifespan $T$ independent of $n$ and $k$.

\subsection{Convergence to the limit}
\begin{proposition}
    There is a subsequence of $\{B_n\}_{n\in\mathbb{N}}$ that converging in $C^{\infty}$ on a compact set that does not contain 0 and in $L^2$ to a solution $B$ of (\ref{the model}).
\end{proposition}
\begin{proof}
    Since $x_{n,k}\leq A^{k}$, from the ODE we know that $\{x_{n,k}(t)\}_{n\in\mathbb{N}}$ is equicontinuous, so we can extract a subsequence such that $x_{n,k}\to x_k$ uniformly for $t\in[T,0]$ and all $k$. From higher order energy estimates, we know that the $C^N$ norms of $W_{n,k}$ are bounded uniformly in $n$, so we can extract a subsequence such that $W_{n,k}\to W_k$ in $C^N$, and thus $B_{n,k}\to B_{k}$ in $C^N$. Denote $B:=\sum_{k=0}^{\infty}B_k$. Since
    \begin{align*}
        \text{supp }B_{n,k}\subset x_{n,k}(r/A)^k\text{supp }W_{n,k}\subset[-(1+2r)r^k,(1+2r)r^k],
    \end{align*}
    we have
    \begin{align*}
        B-B_n=\sum_{k=0}^n(B_k-B_{n,k})+\sum_{k=n+1}^{\infty}B_k\to0
    \end{align*}
    as $n\to\infty$ in $C^N$ away from $(-1/N,1/N)$. Since
    \begin{align*}
        \|\partial_x^mB_{n,k}\|_{L^2}\leq\left(Ar^{3.5-m}\right)^k\|W_{n,k}\|_{L^2},
    \end{align*}
    where $\left(Ar^{3.5-m}\right)^k<1$ for $m<3.5$, we conclude that
    \begin{align*}
        \sum_{j>k}\|B_{n,j}\|_{H^{3.5-}}\to0\quad\text{as}\quad k\to\infty
    \end{align*}
    uniformly in $n$, so $B_n\to B$ in $H^{3.5-}$ as well. Since each $B_n$ is a solution of (\ref{aucxilly equation}) and the dissipation term of (\ref{aucxilly equation}) converges to 0 as $n\to\infty$, so $B$ is a solution of (\ref{the model}). Since $B$ is in $H^{3.5-}$, it satisfies the weak formulation.
\end{proof}

\section{Properties of the solution}
\subsection{H\"older continuity}
\begin{theorem}
    The solution $B(x,t)$ is in $C^{3,s}(\mathbb{R})$ for $|t|$ sufficiently small.
\end{theorem}
\begin{proof}
    We bound the difference $\partial_x^3B(x,t)-\partial_x^3B(y,t)$. There is nothing to prove if $\partial_x^3B(x,t)-\partial_x^3B(y,t)=0$, so we assume without loss of generality that $\partial_x^3B(x,t)\neq0$. If $y$ has a different sign as $x$, by the inequality $a^s+b^s\leq2^{1-s}(a+b)^s$ for $a$, $b\geq0$ and $s\in[0,1]$, it follows that
    \begin{align*}
        \frac{|\partial_x^3B(y,t)-\partial_x^3B(x,t)|}{|x-y|^s}\leq2^{s-1}\max\left(\frac{|\partial_x^3B(x,t)|}{|x|^s},\frac{|\partial_x^3B(y,t)|}{|y|^s}\right).
    \end{align*}
    Since $\lim_{x\to0}\frac{\partial_x^3B(x,t)}{|x|^s}\leq\lim_{x\to0}\frac{\partial_x^3B(x,t)}{|x|}=\partial_x^4B(0)=0$,
    we can assume without loss of generality that $x$ and $y$ have the same sign.\\
    Since $\partial_x^3B(x,t)\neq0$, there is an integer $k$ such that $x\in\text{supp }\partial_x^3B_{n,k}(x,t)$. If $y\notin\cup_{j}\text{supp }\partial_x^3B_{n,j}$, then $\partial_x^3B(y,t)=0$. Let $z$ be one of the two endpoints of $\text{supp }\partial_x^3B_{n,k}$ that is on the same side of $x$ as $y$. Then $\partial_x^3B(z,t)=0$ and $|x-z|\leq|x-y|$, so
    \begin{align*}
        \frac{|\partial_x^3B(y,t)-\partial_x^3B(x,t)|}{|x-y|^s}\leq\frac{|\partial_x^3B_{n,k}(z,t)-\partial_x^3B_{n,k}(x,t)|}{|x-z|^s}.
    \end{align*}
    So we can assume without loss of generality that $y\in\text{supp }\partial_x^3B_{n,j}$ for some $j\in\mathbb{N}$. If $j\neq k$, then the $z$ chosen above is between $x$ and $y$, so $|x-y|\geq|z-y|$ and $|x-z|$. Also, $|\partial_x^3B(y,t)-\partial_x^3B(x,t)|\leq\max(|\partial_x^3B(y,t),|\partial_x^3B(x,t)|)$, so
    \begin{align*}
        \frac{|\partial_x^3B(y,t)-\partial_x^3B(x,t)|}{|x-y|^s}\leq\max\left(\frac{|\partial_x^3B_{n,k}(y,t)-\partial_x^3B_{n,k}(z,t)}{|z-y|^s},\frac{|\partial_x^3B_{n,k}(z,t)-\partial_x^3B_{n,k}(x,t)}{|x-z|^s}\right).
    \end{align*}
    So we can assume without loss of generality that $y\in\text{supp }\partial_x^3B_{n,k}$. Then
    \begin{align*}
        |\partial_x^3B(y,t)-\partial_x^3B(x,t)|\leq\|\partial_x^4B_{n,k}\|_{L^2}\sqrt{|x-y|}\leq\|W_{n,k}\|_{\dot{H}^4}\sqrt{x_k(t)(A/r)^k|x-y|}.
    \end{align*}
    If $t<0$, by Remark \ref{tech remark},
    \begin{align*}
        \liminf_{k\to\infty}\int_t^0x_k(s)ds\geq a
    \end{align*}
    where $a>0$ solves $a=A(1-e^{-a})$. Then as $k\to\infty$,
    \begin{align*}
        x_k(t)\leq x_k(0)e^{(-a+o(1))k}=A^ke^{(-a+o(1))k}
    \end{align*}
    so
    \begin{align*}
        |x-y|\leq4rx_k(t)(r/A)^k\leq r^ke^{(-a+o(1))k}
    \end{align*}
    so for $s\in[0,1/2]$ we have that
    \begin{align*}
        |\partial_x^3B(y,t)-\partial_x^3B(x,t)|\leq \|W_{n,k}\|_{\dot{H}^4}(A^2e^{(2-2s)(-a+o(1))}/r^{2s})^{k/2}|x-y|^s
    \end{align*}
    so $\partial_x^2B\in C^{s-}$ where $A^2e^{-(2-2s)a}=r^{2s}$, or $\ln{A}-(1-s)a=s\ln{r}$, or
    \begin{align*}
        s=\frac{a-\ln{A}}{a-\ln{r}}.
    \end{align*}
    Since $\ln{A}<A-1=A(1-e^{-\ln{A}})$, $a>\ln{A}$ so $s>0$.
\end{proof}
\subsection{Rate of blow up}
\begin{theorem}
    There is a constant $C>0$ such that for all $T<t<0$, $$\frac{C}{|t|}\leq\max{|\partial_x^3B(\cdot,t)|}\leq \frac{1}{C|t|}$$
\end{theorem}
\begin{proof}
    From the ansatz, it follows that
    \begin{align*}
        \max{|\partial_x^3B(\cdot,t)|}=\sup_{k}x_{k}(t)\max|\partial_x^3W_k(\cdot,t)|\sim(\|\partial_x^3\phi\|_{L^{\infty}}+C\epsilon)\sup_{k}x_k(t)
    \end{align*}
    so it suffices to show the same bound for $\sup_kx_k(t)$. Let $k$ be such that $-A^{-k}\leq t\leq -A^{-(k+1)}$. For the upper bound, since $x_k(t)$ is increasing in $t$, we use Lemma \ref{the first preparing lemma} to get
    \begin{align*}
        x_k(t)\leq\frac{1}{|t|}\int_{-T}^0x_k(s)ds\leq\frac{a}{C|t|}
    \end{align*}
    For the lower bound,
    \begin{align*}
        \frac{d}{dt}\ln{x_k(t)}=\sum_{j=0}^{k-1}x_j(t)H\partial_x^2W_{j}(0,t)\leq \|H\phi''\|_{L^{\infty}}\sum_{j=0}^{k-1}A^{j}\leq C\frac{A^k}{A-1}
    \end{align*}
    so
    \begin{align*}
        \ln{x_k}(0)-\ln{x_k(t)}\leq C
    \end{align*}
    so
    \begin{align*}
        x_k(t)\geq e^{-C}A^{k}\geq e^{-C}\frac{1}{A|t|}
    \end{align*}
\end{proof}
\subsection{Non-asymptotic self-similarity}
In this subsection, we show that there is no asymptotically self-similar solution to (\ref{the model}). Consider an ansatz of self-similar solution
\begin{align*}
    \omega=(T-t)^{\alpha}\Omega\left((T-t)^{c}x\right).
\end{align*}
Plugging this ansatz into (\ref{the model}) with $\alpha=-2c-1$ yields the equation
\begin{align*}
    (2c+1)\Omega-cx\Omega_x=2bU\Omega_x+bU_x\Omega,\quad\Omega_x=HU.
\end{align*}
This motivates the following definition.
\begin{definition}
    A solution $\omega$ of (\ref{the model}) is asymptotically self-similar if there is a profile $\Omega$ in some weighted $H^2$ space such that
    \begin{align*}
        -(2c+1)\Omega+cx\Omega_x=2bU\Omega_x+bU_x\Omega,\quad\Omega_x=HU,
    \end{align*}
    i.e.
    \begin{align*}
        (2bU-cx)\Omega_x=(bU_x-2c-1)\Omega
    \end{align*}
    where $-\frac{1}{2}<c$ and for some time-dependent continuous scaling factors $C_w$, $C_l>0$,
    \begin{align*}
        C_w(t)\omega(C_l(t)x,t)\to\Omega
    \end{align*}
    in some weighted $L^2$ norm as $t$ approaches the blow-up time.
\end{definition}
However, using the dynamic rescaling formulation, one can find that there is no asymptotic self-similar profile that satisfies the above definition.
Denote
\begin{align*}
    \tilde{\omega}(x,t)=C_w(t)\omega(C_l(t)x,t),
\end{align*}
where $\omega$ is a solution to (\ref{the model}). Plugging $\tilde{\omega}$ in (\ref{the model}) yields
\begin{align*}
    \left(2bC_w^{-1}C_l^{-2}\tilde{u}-\dot{C}_lC_l^{-1}x\right)\partial_x\tilde{\omega}=\left(bC_w^{-1}C_l^{-2}\tilde{u}_x+C_w^{-1}\right)\tilde{\omega}.
\end{align*}
Passing to the limit $t\to 0$ and requiring it to satisfy the definition of an asymptotic self-similar solution, one needs
\begin{align*}
    Cw^{-1}C_l^{-2}\to1\text{ and } C_w^{-1}\to-(2c+1)
\end{align*}
This implies that
\begin{align*}
    C_l^{-2}\to-\frac{1}{2c+1}<0
\end{align*}
which is impossible since $c>-\frac{1}{2}$. Thus the blow-up solution cannot be asymptotically self-similar.

\appendix
\section{Auxiliary estimates}
\begin{lemma}
    $E_1\leq C\|W_{n,k}-\phi\|_{\dot{H}^4}^2$
\end{lemma}
\begin{proof}
    Integration by parts yields
    \begin{align*}
        E_1
        =&\sum_{l=0}^4\frac{2b}{x_{n,k}}\left(\frac{A}{r}\right)^kC_4^lx_{n,k}^l(t)\left(\frac{r}{A}\right)^{lk}\times\\
        &\times\int\partial_x^4(W_{n,k}-\phi)\partial_x^l\partial_xHB_n\left(xx_{n,k}(t)\left(\frac{r}{A}\right)^k,t\right)\partial_x^{5-l}(W_{n,k}-\phi)dx\\
        &:=\sum_{l=0}^4E_{1,l}(t).
    \end{align*}
    \begin{align*}
       E_{1,0}
       =&\frac{2b}{x_{n,k}}\left(\frac{A}{r}\right)^k\int\partial_x^4(W_{n,k}-\phi)\partial_xHB_n\left(xx_{n,k}(t)\left(\frac{r}{A}\right)^k,t\right)\partial_x^{5}(W_{n,k}-\phi)dx\\
       =&-b\int\partial_x^4(W_{n,k}-\phi)\partial_x^2HB_n\left(xx_{n,k}(t)\left(\frac{r}{A}\right)^k,t\right)\partial_x^{4}(W_{n,k}-\phi)dx\\
       \leq&C_{11}(b,r,\epsilon)\|W_{n,k}-\phi\|_{\dot{H}^4}
    \end{align*}
    When $l\geq1$, by Lemma \ref{for closing bootstrap} and Poincar\'e inequalities, we have
    \begin{align*}
        E_{1,l}
        =&2bC_4^{l}x_{n,k}^{l-1}\left(\frac{r}{A}\right)^{(l-1)k}\times\\
        &\times\int\partial_x^4(W_{n,k}-\phi)\partial_x^l\partial_xHB_n\left(xx_{n,k}(t)\left(\frac{r}{A}\right)^k,t\right)\partial_x^{5-l}(W_{n,k}-\phi)dx\\
        &\leq C\|W_{n,k}-\phi\|_{\dot{H}^4}^2
    \end{align*}
\end{proof}
\begin{lemma}
   $E_2(t)\leq C\|W_{n,k}-\phi\|_{\dot{H}^4}$
\end{lemma}
\begin{proof}
    Note that
    \begin{align*}
        E_2(t)
        =&2b\frac{1}{x_{n,k}}\left(\frac{A}{r}\right)^k\int\partial_x^8(W_{n,k}-\phi)\partial_xHB_n\left(xx_{n,k}(t)\left(\frac{r}{A}\right)^k,t\right)\phi'dx\\
        =&2b\frac{1}{x_{n,k}}\left(\frac{A}{r}\right)^k\int\partial_x^4(W_{n,k}-\phi)\partial_x^4\left(\partial_xHB_n\left(xx_{n,k}(t)\left(\frac{r}{A}\right)^k,t\right)\phi'\right)dx.
    \end{align*}
    Then by Kato-Ponce inequality and Lemma \ref{for closing bootstrap}, we have
    \begin{align*}
        E_2(t)\leq
        &Cx_{n,k}^3\left(\frac{r}{A}\right)^3\left\|\partial_x^5HB_n\left(xx_{n,k}(t)\left(\frac{r}{A}\right)^k,t\right)\right\|_{L^2}\|\phi'\|_{L^{\infty}}\|W_{n,k}-\phi\|_{\dot{H}^4}+\\
        &+C\left\|\partial_xHB_n\left(xx_{n,k}(t)\left(\frac{r}{A}\right)^k,t\right)\right\|_{L^{\infty}}\left\|\frac{d^5\phi}{dx^5}\right\|_{L^{2}}\|W_{n,k}-\phi\|_{\dot{H}^4}\\
        \leq&C\|W_{n,k}-\phi\|_{\dot{H}^4}
    \end{align*}
\end{proof}
\begin{lemma}
    $E_3(t)\leq C\|W_{n,k}-\phi\|_{\dot{H}^4}$
\end{lemma}
\begin{proof}
    By Lemma \ref{for closing bootstrap} and integration by parts,
    \begin{align*}
        E_3(t)
        =&\int\partial_x^8(W_{n,k}-\phi)\frac{x\dot{x}_{n,k}(t)}{x_{n,k}(t)}\partial_x(W_{n,k}-\phi)dx\\
        =&-\frac{1}{2}\int\partial_x^4(W_{n,k}-\phi)\frac{\dot{x}_{n,k}(t)}{x_{n,k}(t)}\partial_x^4(W_{n,k}-\phi)dx+\\
        &+4\int\partial_x^4(W_{n,k}-\phi)\frac{\dot{x}_{n,k}(t)}{x_{n,k}(t)}\partial_x^4(W_{n,k}-\phi)dx\\
        \leq&C\|W_{n,k}-\phi\|_{\dot{H}^4}^2
    \end{align*}
\end{proof}
\begin{lemma}
    $E_4(t)\leq C\|W_{n,k}-\phi\|_{\dot{H}^4}$
\end{lemma}
\begin{proof}
    By Lemma \ref{for closing bootstrap} and integration by parts,
    \begin{align*}
        E_4(t)
        =&\int\partial_x^8(W_{n,k}-\phi)\frac{x\dot{x}_{n,k}(t)}{x_{n,k}(t)}\phi'dx\\
        =&\int\partial_x^4(W_{n,k}-\phi)\frac{x\dot{x}_{n,k}(t)}{x_{n,k}(t)}\frac{d^5\phi}{dx^5}dx+4\int\partial_x^4(W_{n,k}-\phi)\frac{\dot{x}_{n,k}(t)}{x_{n,k}(t)}\frac{d^4\phi}{dx^4}dx\\
        \leq&C\|W_{n,k}-\phi\|_{\dot{H}^4}
    \end{align*}
\end{proof}
\begin{lemma}
    $E_5(t)\leq C\|W_{n,k}-\phi\|_{\dot{H}^4}$
\end{lemma}
\begin{proof}
    \begin{align*}
        E_5
        =&\int\partial_x^8(W_{n,k}-\phi)b\partial_x^2HB_{n}\left(xx_{n,k}(t)\left(\frac{r}{A}\right)^k,t\right)(W_{n,k}-\phi)dx\\
        =&b\int\partial_x^4(W_{n,k}-\phi)\partial_x^4\left(\partial_x^2HB_{n}\left(xx_{n,k}(t)\left(\frac{r}{A}\right)^k,t\right)(W_{n,k}-\phi)\right)dx
    \end{align*}
    Then by Lemma \ref{for higher energy estimates} and Kato-Ponce inequality,
    \begin{align*}
        E_5(t)\leq C\|W_{n,k}-\phi\|_{\dot{H}^4}
    \end{align*}
\end{proof}
\begin{lemma}
    $E_6(t)\leq C\|W_{n,k}-\phi\|_{\dot{H}^4}$
\end{lemma}
\begin{proof}
    \begin{align*}
        E_6
        =&\int\partial_x^8(W_{n,k}-\phi)b\partial_x^2HB_{n}\left(xx_{n,k}(t)\left(\frac{r}{A}\right)^k,t\right)\phi dx\\
        =&b\int\partial_x^4(W_{n,k}-\phi)\partial_x^4\left(\partial_x^2HB_{n}\left(xx_{n,k}(t)\left(\frac{r}{A}\right)^k,t\right)\phi \right)dx\\
    \end{align*}
    Then by Lemma \ref{for higher energy estimates} and Kato-Ponce inequality,
    \begin{align*}
        E_6(t)\leq C\|W_{n,k}-\phi\|_{\dot{H}^4}
    \end{align*}
\end{proof}
\begin{lemma}
    \begin{align*}
        E_7&\leq C\|W_{n,k}-\phi\|_{\dot{H}^4}^2\\
        E_8&\leq C\|W_{n,k}-\phi\|_{\dot{H}^4}
    \end{align*}
\end{lemma}
\begin{proof}
    These two inequalities are a direct consequence of the Lemma \ref{for closing bootstrap}.
\end{proof}

\begin{center}
    $\mathbf{Acknowledgment}$
\end{center}
I would like to express my gratitude to Mimi Dai for the open problem and stimulating discussions.

\bibliography{Citations}
\bibliographystyle{plain}
\end{document}